\documentclass[12pt,a4paper]{article}%
\usepackage[utf8]{inputenc}
\usepackage{amsmath}
\usepackage{amsfonts}
\usepackage{amssymb}
\usepackage{graphicx}
\usepackage[space]{grffile}%
\providecommand{\U}[1]{\protect\rule{.1in}{.1in}}
\newtheorem{theorem}{Theorem}
\newtheorem{theorem*}{Example}

\newtheorem{conjecture}[theorem]{Conjecture}

\newtheorem{lemma}[theorem]{Lemma}

\newenvironment{proof}[1][Proof]{\noindent\textbf{#1.} }{\ \hfill \rule{0.5em}{0.5em}\bigskip}
\graphicspath{{D:/Dropbox/Riste-Sedlar/Totally Irregular/Slike/}}

\begin{document}

\title{Remarks on the Local Irregularity Conjecture}
\author{Jelena Sedlar$^{1,3}$,\\Riste \v Skrekovski$^{2,3}$ \\[0.3cm] {\small $^{1}$ \textit{University of Split, Faculty of civil
engineering, architecture and geodesy, Croatia}}\\[0.1cm] {\small $^{2}$ \textit{University of Ljubljana, FMF, 1000 Ljubljana,
Slovenia }}\\[0.1cm] {\small $^{3}$ \textit{Faculty of Information Studies, 8000 Novo
Mesto, Slovenia }}\\[0.1cm] }
\maketitle

\begin{abstract}
A locally irregular graph is a graph in which the end-vertices of every edge
have distinct degrees. A locally irregular edge coloring of a graph $G$ is any
edge coloring of $G$ such that each of the colors induces a locally irregular
subgraph of $G$. A graph $G$ is colorable if it admits a locally irregular
edge coloring. The locally irregular chromatic index of a colorable graph $G$,
denoted by $\chi_{%
\rm{irr}%
}^{\prime}(G)$, is the smallest number of colors used by a locally irregular
edge coloring of $G$. The Local Irregularity Conjecture claims that all
graphs, except odd length path, odd length cycle and a certain class of cacti,
are colorable by $3$ colors. As the conjecture is valid for graphs with large
minimum degree and all non-colorable graphs are vertex disjoint cacti, we take
direction to study rather sparse graphs. In this paper, we give a cactus graph
$B$ which contradicts this conjecture, i.e. $\chi_{%
\rm{irr}%
}^{\prime}(B)=4$. Nevertheless, we show that the conjecture holds for
unicyclic graphs and cacti with vertex disjoint cycles.

\end{abstract}



\section{Introduction}

All graphs mentioned in this paper are considered to be simple and finite. An
edge coloring of a graph is \emph{neighbor-sum-distinguishing} if any two
neighboring vertices differ in the sum of the colors of the edges incident to
them. This notion was first introduced in \cite{Karonski} and the following
conjecture was proposed there.

\begin{conjecture}
[1-2-3 Conjecture]Every graph $G$ without isolated edges admits a
neighbor-sum-distinguishing edge-coloring with the colors $\{1,2,3\}$.
\end{conjecture}

This conjecture attracted a lot of interest \cite{AdarioBerry123,
Kalkowski123, Przibilo12, Przibilo123, Wang123} and for a survey we reffer the
reader to \cite{SeamoneSurvey}. The best upper bound is that every graph
without isolated edges admits a neighbor-sum-distinguishing edge-coloring with
five colors \cite{KaronskiSol}, but the 1-2-3 Conjecture remains open.

This variant of edge coloring and the 1-2-3 Conjecture motivated introduction
of similar variants of edge coloring. A \emph{locally irregular} graph is any
graph in which the two end-vertices of every edge differ in degree. A
\emph{locally irregular }$k$\emph{-edge coloring}, or $k$\emph{-liec} for
short, is any edge coloring of $G$ with $k$ colors such that every color
induces a locally irregular subgraph of $G$. This variant of edge coloring was
introduced in \cite{Baudon5}. A third related edge coloring variant is is the
\emph{neighbor multiset-distinguishing edge-coloring}, where neighboring
vertices must have assigned distinct multisets of colors on incident edges. In
\cite{AdarioBerry}, it was established that every graph without isolated edges
admits the neighbor multiset-distinguishing edge-coloring with four colors.
Notice that every locally irregular edge coloring is also a neighbor
multiset-distinguishing edge-coloring, but the reverse does not have to hold.

In this paper we focus our attention to locally irregular edge colorings
exclusively, and we say a graph is \emph{colorable} if it admits such a
coloring. The \emph{locally irregular chromatic index} of a colorable graph
$G,$ denoted by $\chi_{%
\rm{irr}%
}^{\prime}(G)$ is the smallest $k$ such that $G$ admits a $k$-liec. In
\cite{Baudon5}, the family of graphs $\mathfrak{T}$ has been defined as follows:

\begin{itemize}
\item $\mathfrak{T}$ contains the triangle $K_{3}$,

\item if $G$ is a graph from $\mathfrak{T}$, then a graph $H$ obtained from
$G$ by identifying a vertex $v\in V(G)$ of degree $2,$ which belongs to a
triangle of $G,$ with an end-vertex of an even length path or with an end
vertex of an odd length path such that the other end vertex of that path is
identified with a vertex of a triangle.
\end{itemize}

Note that every graph $G\in\mathfrak{T}$ has odd size. A \emph{cactus graph}
is any graph in which cycles are edge disjoint. Notice that $\mathfrak{T}$ is
a special family of cacti. Also, if we imagine triangles to be vertices and
paths attached to vertices of a triangle as edges, we might informally say
that $G$ has tree-like structure. For the sake of simplicity, we define a
broader family $\mathfrak{T}^{\prime}$ as the family obtained from
$\mathfrak{T}$ by introducing to it all odd length paths and all odd length
cycles. Notice that $\mathfrak{T}^{\prime}$ is a subclass of vertex-disjoint
cactus graphs. It was established in \cite{Baudon5} that a connected graph $G$
is not colorable if and only if $G\in\mathfrak{T}^{\prime}$. Also, the
following conjecture on the irregular chromatic index was proposed.

\begin{conjecture}
[Local Irregularity Conjecture]\label{Con_nonTareColorable}For every connected
graph $G\not \in \mathfrak{T}^{\prime}$, it holds that $\chi_{%
\rm{irr}%
}^{\prime}(G)\leq3.$
\end{conjecture}

Let us mention some of the results related to Conjecture
\ref{Con_nonTareColorable}. For general graphs it was first established
$\chi_{%
\rm{irr}%
}^{\prime}(G)\leq328$ \cite{Bensmail}, then it was lowered to $\chi_{%
\rm{irr}%
}^{\prime}(G)\leq220$ \cite{Luzar}. Fore some special classes of graphs
Conjecture \ref{Con_nonTareColorable} is shown to hold, namely for trees
\cite{Baudon6}, graphs with minimum degree at least $10^{10}$ \cite{Przibilo},
$k$-regular graphs where $k\geq10^{7}$ \cite{Baudon5}.

In this paper we will show that every unicyclic graph $G$ which does not
belong to $\mathfrak{T}^{\prime}$ admits a $3$-liec, thus establishing that
Conjecture \ref{Con_nonTareColorable} holds for unicyclic graphs. We will
further extend this result to cactus graphs with vertex disjoint cycles.
Finally, we will provide an example of a colorable graph $B$ with $\chi_{%
\rm{irr}%
}^{\prime}(B)=4$ showing thus that Conjecture \ref{Con_nonTareColorable} does
not hold in general. Possibly this is the only counterexample to the conjecture.

\section{Revisiting the trees}

Since a unicyclic graph is obtained from a tree by adding a single edge to it,
we first need to introduce the notation and several important results for
trees from \cite{Baudon6}. Also, we will establish several auxiliary results
for trees, which will be useful throughout the paper.

First, a \emph{shrub} is any tree rooted at a leaf. The only edge in a shrub
$G$ incident to the root we will call the \emph{root edge} of $G$. An
\emph{almost locally irregular }$k$\emph{-edge coloring} of a shrub $G,$ or
$k$\emph{-aliec} for short, is an edge coloring of $G$ which is either
$k$-liec or a coloring in which only the root edge is locally irregular
(notice that in this case the root edge is an isolated edge of its color i.e.
it is not adjacent to any other edge of the same color). A \emph{proper }%
$k$\emph{-aliec} is $k$-aliec which is not a $k$-liec. The following results
for trees were established in \cite{Baudon6}.

\begin{theorem}
\label{Tm_BaudonSchrub}Every shrub admits a $2$-aliec.
\end{theorem}

\begin{theorem}
\label{Tm_BaudonTree}For every colorable tree $T,$ it holds that $\chi_{%
\rm{irr}%
}^{\prime}(T)\leq3.$ Moreover, $\chi_{%
\rm{irr}%
}^{\prime}(T)\leq2$ if $\Delta(T)\geq5.$
\end{theorem}

If an edge coloring uses at most three colors, we will denote those colors by
$a,b,c.$ A $1$-liec (resp. $2$-liec, $3$-liec) of a graph $G$ will be denoted
by $\phi_{a}(G)$ (resp. $\phi_{a,b}(G),$ $\phi_{a,b,c}(G)$). A $2$-aliec of a
shrub $G$ will be denoted by $\phi_{a,b}(G)$ where $a$ is the color of the
root edge in $G$. Let $a,b,c,d$ be four colors, if $\phi_{a,b}(G)$ is a
$2$-liec of $G$ in colors $a$ and $b,$ then $2$-liec $\phi_{c,d}(G)$ of $G$ in
colors $c$ and $d$ is obtained from $\phi_{a,b}(G)$ by \emph{replacing} colors
$a$ and $b$ for $c$ and $d$ respectively, i.e. $\phi_{c,d}(e)=c$ if and only
if $\phi_{a,b}(e)=a.$ Particularly, $2$-(a)liec $\phi_{b,a}(G)$ is called the
\emph{inversion} of the $2$-(a)liec $\phi_{a,b}(G),$ where colors $a$ and $b$
are replaced. Moreover, let $\phi_{a,b,c}^{i}$ be an edge coloring of a graph
$G_{i}$, for $i=1,\ldots,k,$ and let $\cap_{i=1}^{k}E(G_{i})=\phi.$ For a
graph $G$ such that $E(G)=\cup_{i=1}^{k}E(G_{i}),$ by $\sum_{i=1}^{k}%
\phi_{a,b,c}^{i}$ we will denote the edge coloring of $G$ such that an edge
$e$ is colored by $\phi_{a,b,c}^{i}(e)$ if and only if $e\in E(G_{i}).$

For any color of the edge coloring $\phi_{a,b,c},$ say $a$, we define the
$a$\emph{-degree} of a vertex $v\in V(G)$ as the number of edges incident to
$v$ which are colored by $a$. The $a$-degree of a vertex $v$ is denoted by
$d_{G}^{a}(v).$ Assume that a vertex $v\in V(G)$ has $k$ neighbors
$w_{1},\ldots,w_{k}$ such that each $vw_{i}$ is colored by $a.$ Then the
sequence $d_{G}^{a}(w_{1}),\ldots,d_{G}^{a}(w_{k})$ is called the
$a$\emph{-sequence} of the vertex $v$. We usually assume that neighbors of $v$
are denoted so that the $a$-sequence is non-increasing.

Throughout the paper we will use the technique of finding a $2$-liec for trees
introduced in \cite{Baudon6}. Namely, if $T$ is a tree with maximum degree $5$
or more, then $T$ admits a $2$-liec according to Theorem \ref{Tm_BaudonTree}.
Otherwise, if the maximum degree of $T$ is at most $4,$ let $v$ be a vertex
from $T$ and $w_{1},\ldots,w_{k}$ all the neighbors of $v$ for $k\leq4.$
Notice that $T$ consists of $k$ shrubs $T_{i}$ starting at $v,$ let $T_{i}$
denote a shrub with the root edge $vw_{i}$ and let $\phi_{a,b}^{i}$ denote a
$2$-aliec of $T_{i}$ which exists according to Theorem \ref{Tm_BaudonSchrub}.
Recall that $\phi_{a,b}^{i}(vw_{i})=a$ for every $i\leq k.$ The coloring
$\phi_{a,b}=\sum_{i=1}^{k}\phi_{a,b}^{i}$ is called a \emph{shrub based} edge
coloring of $T.$ We say that a shrub based coloring $\phi_{a,b}$ is
\emph{inversion resistant} if neither $\phi_{a,b}$ is a $2$-liec of $T$ nor
any of the colorings which can be obtained from $\phi_{a,b}$ by color
inversion in some of the shrubs $T_{i}.$ Let us now introduce the following
lemma which stems from the technique used in \cite{Baudon6}.

\begin{lemma}
\label{Lemma_inversion}Let $T$ be a tree with $\Delta(T)\leq4$ and $v$ a
vertex from $T$ of degree $k.$ Let $T_{1},\ldots,T_{k}$ be all the shrubs of
$T$ rooted at $v$ and let $\phi_{a,b}^{i}$ be a $2$-aliec of $T_{i}.$ If
$\phi_{a,b}^{i}$ is a $2$-liec of $T_{i}$ for every $i=1,\ldots,k,$ then the
shrub based coloring $\phi_{a,b}=\sum_{i=1}^{k}\phi_{a,b}^{i}$ can be
inversion resistant in two cases only:

\begin{itemize}
\item if $d_{T}(v)=3$ and the $a$-sequence of $v$ by $\phi_{a,b}$ is $3,2,2$;

\item if $d_{T}(v)=4$ and the $a$-sequence of $v$ by $\phi_{a,b}$ is $4,3,3,2$.
\end{itemize}
\end{lemma}

\begin{proof}
If $d_{T}(v)=1,$ then the shrub based coloring of $T$ equals $\phi_{a,b}^{1},$
which is $2$-liec. If $d_{T}(v)=2,$ then $\phi_{a,b}^{1}+\phi_{b,a}^{2}$ would
be a $2$-liec of $T.$

If $d_{T}(v)=3,$ then the $a$-degree of $v$ by $\phi_{a,b}$ is $3,$ and
inverting colors in one of the shrubs would decrease the $a$-degree of $v$ to
$2$. Therefore, the $a$-sequence of $v$ by $\phi_{a,b}$ must contain $3$ and
$2.$ Considering the two possibilities $4,3,2,$ and $3,3,2$, we see that
$\phi_{a,b}^{1}+\phi_{a,b}^{2}+\phi_{b,a}^{3}$ would be $2$-liec in both of
them. The only remaining possibility is $3,2,2,$ and it is inversion resistant.

Finally, assume $d_{T}(v)=4.$ By a similar consideration as above, we see that
the $a$-sequence of $v$ by $\phi_{a,b}$ must contain $4,$ $3,$ and $2.$
Therefore, we must consider the possibilities $4,4,3,2,$ then $4,3,3,2,$ then
$4,3,2,2.$ It is easily seen that only in the case $4,3,3,2$, the shrub based
coloring $\phi_{a,b}$ is inversion resistant.
\end{proof}

A \emph{spidey} is a tree with radius at most two which consists of a central
vertex $u$ of degree at least $3$ and the remaining vertices have degree at
most $2$ and are at distance at most $2$ from $u$. Notice that every spidey is
locally irregular, hence it admits a $1$-liec. We say that a vertex $v$ of a
spidey $G$ is a \emph{short leg} if it is a leaf which is a neighbor of the
central vertex of $G.$

\begin{figure}[h]
\begin{center}
$%
\begin{array}
[c]{ll}%
\text{a) \raisebox{-1\height}{\includegraphics[scale=0.6]{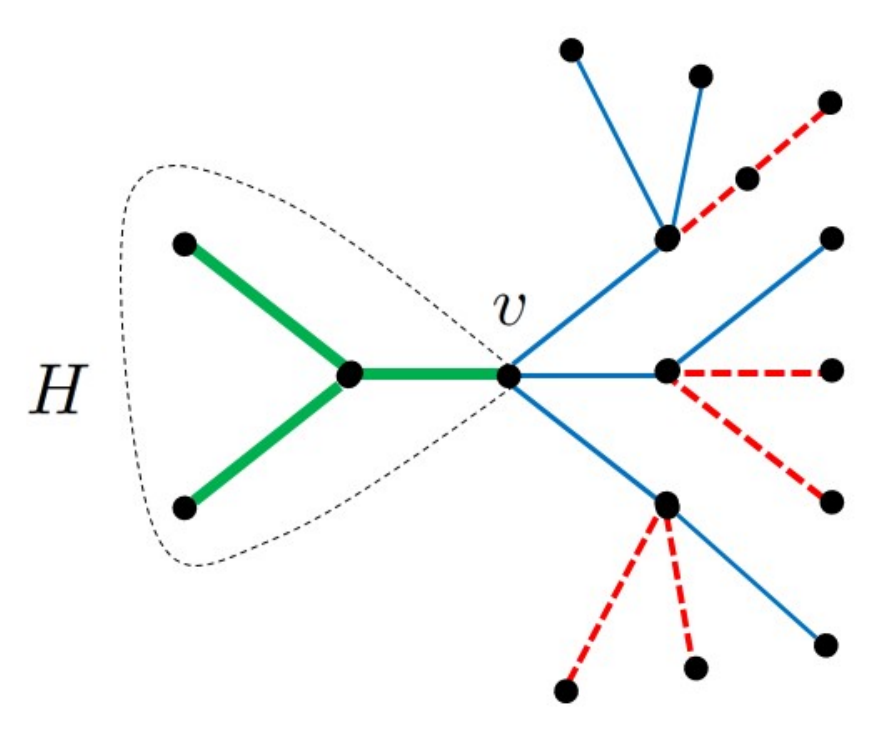}}} &
\text{b) \raisebox{-1\height}{\includegraphics[scale=0.6]{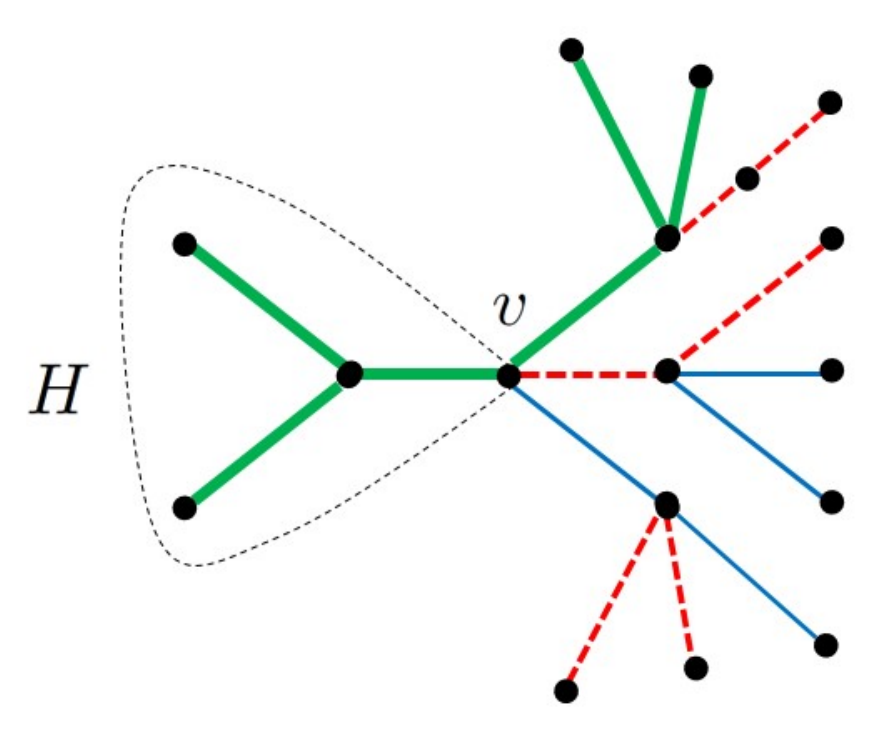}}}%
\end{array}
$
\end{center}
\caption{A graph $G=H+K$ and a vertex $v$ of degree $3$ in $K:$ a) the
coloring $\phi_{c}^{0}+\phi_{a,b}^{1}+\phi_{a,b}^{2}+\phi_{a,b}^{3}$ is not a
$3$-liec of $G$, b) the coloring $\phi_{c}^{0}+\phi_{c,b}^{1}+\phi_{b,a}%
^{2}+\phi_{a,b}^{3}$ is a $3$-liec of $G.$}%
\label{Fig_SpideyGlue3}%
\end{figure}

\begin{lemma}
\label{Lemma_spidey}Let $H$ be a spidey with a short leg $v$ and let $K$ be a
tree. Let $G$ be a graph obtained from $H$ and $K$ by identifying the vertex
$v$ with a vertex from $K.$ Then $G$ admits a $3$-liec such that all edges of
$E(H)$ are colored by a same color.
\end{lemma}

\begin{proof}
Since $H$ is a spidey, $H$ admits a $1$-liec, say $\phi_{c}^{0}$. Assume first
that a tree $K$ is not colorable, i.e. $K$ is an odd length path. This implies
there exists in $K$ an edge $e_{v}$ incident with $v$ such that $K-e_{v}$ is a
collection of even paths which therefore admits $2$-liec $\phi_{a,b}^{1}.$ The
edge coloring $\phi_{a,b,c}$ of $G$ defined by%
\[
\phi_{a,b,c}(e)=\left\{
\begin{array}
[c]{ll}%
c & \text{if }e=e_{v},\\
\phi_{a,b}^{1}(e) & \text{if }e\in E(K)\backslash\{e_{v}\},\\
\phi_{c}^{0}(e) & \text{if }e\in E(H),
\end{array}
\right.
\]
is a $3$-liec of $G.$

Assume now that $K$ is a colorable tree. If $K$ admits a $2$-liec $\phi
_{a,b}^{1},$ then $\phi_{c}^{0}+\phi_{a,b}^{1}$ is a $3$-liec of $G$ with the
desired property. So, we may assume $K$ is a colorable tree which does not
admit a $2$-liec. Theorem \ref{Tm_BaudonTree} implies $\Delta(T)\leq4.$ Let
$d_{K}(v)=k\leq4$ and let $T_{1},\ldots,T_{k}$ be all the shrubs of $K$ rooted
at $v.$ By Theorem \ref{Tm_BaudonSchrub}, each shrub $T_{i}$ admits a
$2$-aliec $\phi_{a,b}^{i},$ where without loss of generality we may assume
that $\phi_{a,b}^{i}$ is a proper $2$-aliec if and only if $i\leq l.$ We
distinguish the following four cases with respect to $l.$

\medskip\noindent\textbf{Case 1:} $l\geq3.$ Notice that $l\in\{3,4\}$ and
$l\leq k\leq4.$ If $l=3$ and $k=4,$ then $\phi_{a,b}^{1}+\phi_{a,b}^{2}%
+\phi_{a,b}^{3}+\phi_{b,a}^{4}$ would be a $2$-liec of $K,$ a contradiction.
Otherwise, the shrub based coloring $\phi_{a,b}=\sum_{i=1}^{k}\phi_{a,b}^{i}$
would be a $2$-liec of $K,$ again a contradiction.

\medskip\noindent\textbf{Case 2:} $l=2.$ If $k=2,$ then $\phi_{a,b}^{1}%
+\phi_{a,b}^{2}$ is a $2$-liec of $K,$ a contradiction. If $k=3,$ then
$\phi_{a,b}^{1}+\phi_{a,b}^{2}+\phi_{b,a}^{3}$ is a $2$-liec of $K,$ a
contradiction. If $k=4,$ then let $w_{1},\ldots,w_{4}$ be all the neighbors of
$v$ in $K.$ The shrub based coloring $\phi_{a,b}=\sum_{i=1}^{k}\phi_{a,b}^{i}$
is not a $2$-liec only if the $a$-degree of $w_{3}$ or $w_{4}$ by $\phi_{a,b}$
is $4.$ Without loss of generality we may assume that $a$-degree of $w_{3}$ by
$\phi_{a,b}$ is $4,$ but then $\phi_{a,b}^{1}+\phi_{a,b}^{2}+\phi_{a,b}%
^{3}+\phi_{b,a}^{4}$ is a $2$-liec of $K,$ a contradiction.

\medskip\noindent\textbf{Case 3:} $l=1.$ In this case $T_{1}$ is the only
shrub with a proper $2$-aliec $\phi_{a,b}^{1}.$ Let $w_{1}$ be the neighbor of
$v$ in $T_{1},$ we define the coloring $\phi_{a,b}^{\prime}$ of $K$ as follows%
\[
\phi_{a,b,c}^{\prime}(e)=\left\{
\begin{array}
[c]{ll}%
c & \text{if }e=vw_{1},\\
\sum_{i=1}^{k}\phi_{a,b}^{i}(e) & \text{if }e\in E(K)\backslash\{vw_{1}\}.
\end{array}
\right.
\]
Notice that $\phi_{a,b,c}^{\prime}$ is not a liec of $K,$ but $\phi_{c}%
^{0}+\phi_{a,b,c}^{\prime}$ is a $3$-liec of $G=H+K$ with the desired property
that all edges of $H$ are colored by a same color, in this case $c$.

\medskip\noindent\textbf{Case 4:} $l=0.$ Notice that in this case Lemma
\ref{Lemma_inversion} applies on $K$ and $v.$ Therefore, the only cases when
$K$ does not admit a $2$-liec are: i) $d_{K}(v)=3$ and the $a$-sequence of $v$
by the shrub based coloring $\phi_{a,b}=\sum_{i=1}^{k}\phi_{a,b}^{i}$ is
$3,2,2,$ or ii) $d_{K}(v)=4$ and the $a$-sequence of $v$ by $\phi_{a,b}$ is
$4,3,3,2.$ In the first case the coloring $\phi_{c}^{0}+\phi_{c,b}^{1}%
+\phi_{b,a}^{2}+\phi_{a,b}^{3}$ is a $3$-liec of $G$ such that $E(H)$ is
colored by the same color $c$, as it is illustrated in Figure
\ref{Fig_SpideyGlue3}. In the other case, the coloring $\phi_{c}^{0}%
+\phi_{a,b}^{1}+\phi_{c,b}^{2}+\phi_{a,b}^{3}+\phi_{b,a}^{4}$ is a $3$-liec of
$G$ such that $E(H)$ is colored by a same color, as it is illustrated in
Figure \ref{Fig_spideyGlue4}.
\end{proof}

\begin{figure}[h]
\begin{center}
$%
\begin{array}
[c]{ll}%
\text{a) \raisebox{-1\height}{\includegraphics[scale=0.6]{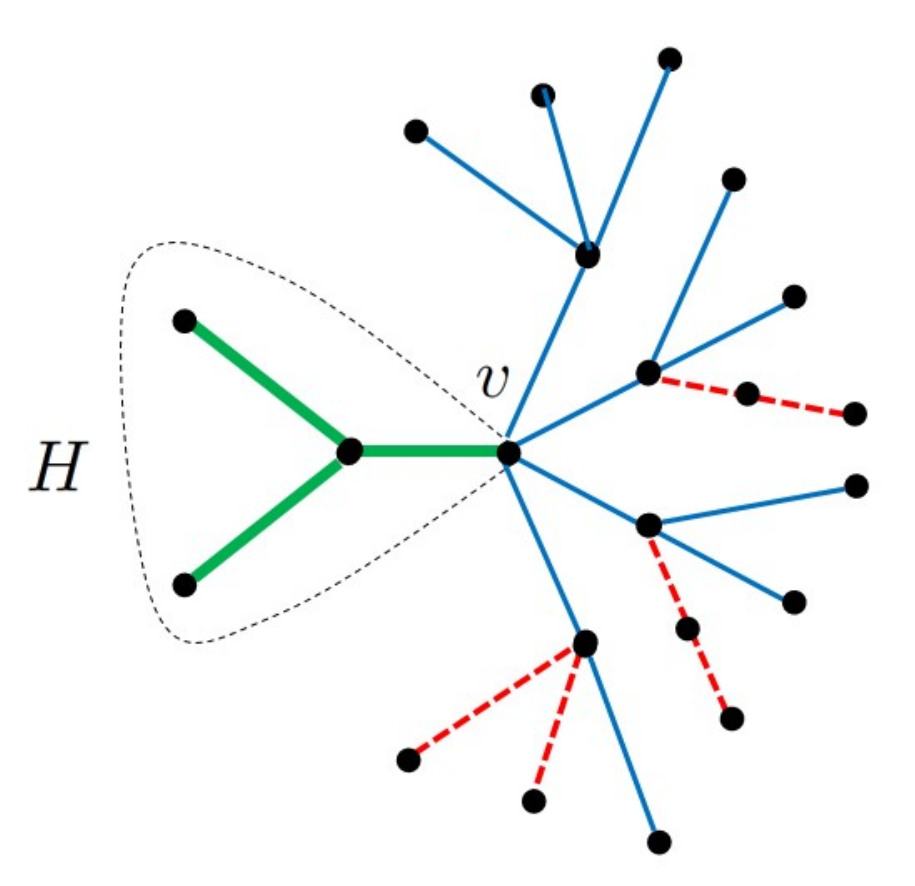}}} &
\text{b) \raisebox{-1\height}{\includegraphics[scale=0.6]{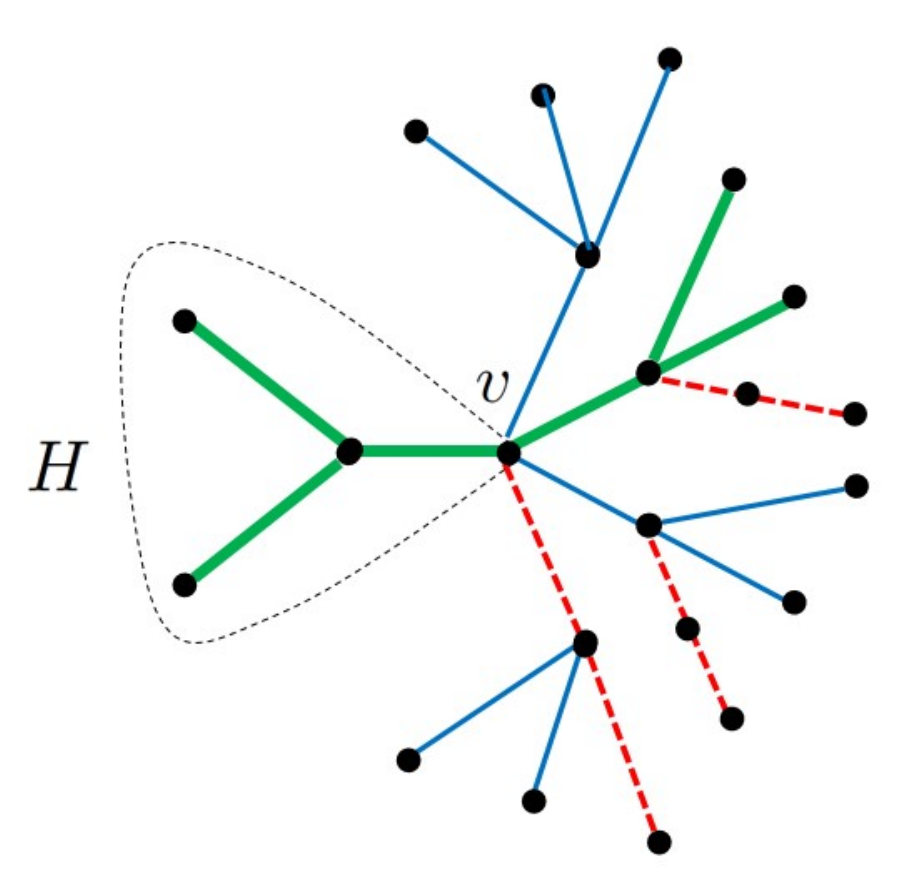}}}%
\end{array}
$
\end{center}
\caption{A graph $G=H+K$ and a vertex $v$ of degree $4$ in $K:$ a) the
coloring $\phi_{c}^{0}+\phi_{a,b}^{1}+\phi_{a,b}^{2}+\phi_{a,b}^{3}+\phi
_{a,b}^{4}$ is not a $3$-liec of $G$, b) the coloring $\phi_{c}^{0}+\phi
_{a,b}^{1}+\phi_{c,b}^{2}+\phi_{a,b}^{3}+\phi_{b,a}^{4}$ is a $3$-liec of
$G.$}%
\label{Fig_spideyGlue4}%
\end{figure}

\section{Unicyclic graphs}

In this section we will establish Conjecture \ref{Con_nonTareColorable} for
unicyclic graphs. It is already known that there exist colorable unicyclic
graphs which do not admit $2$-liec, but require $3$ colors in order for edge
coloring to be locally irregular, namely cycles of length $4k+2,$ for
$k\in\mathbb{N}$. We will show that such cycles are not an isolated family of
unicyclic graphs that require three colors. The main result for unicyclic
graphs is established through the following two lemmas in which we will
consider separately cases whether the cycle of $G$ is a triangle or not.

\begin{lemma}
Let $G$ be a unicyclic graph with the unique cycle being a triangle. If
$G\not \in \mathfrak{T}^{\prime}\mathfrak{,}$ then $\chi_{%
\rm{irr}%
}^{\prime}(G)\leq3.$
\end{lemma}

\begin{proof}
Let $C=u_{1}u_{2}u_{3}$ be the $3$-cycle in $G$, let $T_{i}$ denote the
connected component of $G-E(C)$ which contains $u_{i}.$ Since $G\not \in
\mathfrak{T}^{\prime},$ there must exist a vertex $u_{i}$ on $C$ such that
$T_{i}$ is not a pendant even length path, say it is $u_{1}.$ Let $G_{1}%
=T_{1}+u_{1}u_{2}$ and let $G_{0}=G-E(G_{1}).$ First notice that both $G_{0}$
and $G_{1}$ are trees and that $E(G)=E(G_{0})\cup E(G_{1}).$ Since $T_{1}$ is
not a pendant path of even length, it follows that $G_{1}$ is not an odd
length path, hence it is colorable. Let $\phi_{a,b,c}^{1}$ be a $3$-liec of
$G_{1}.$ Without loss of generality we may assume that $\phi_{a,b,c}^{1}%
(u_{1}u_{2})=c.$ Let $H$ be the subgraph of $G_{1}$ induced by all edges
incident to $u_{1}$ in $G_{1}.$ We may assume $\phi_{a,b,c}^{1}(e)=c$ for
every $e\in E(H),$ namely if $d_{H}(u_{1})=2$ this follows from the local
irregularity of $\phi_{a,b,c}^{1},$ otherwise it follows from Lemma
\ref{Lemma_spidey} applied on $H$ and every component of $G_{1}-E(H)$ repeatedly.

Let us now consider the graph $G_{0}$ and notice that it is a shrub rooted at
$u_{1}$ with the root edge $u_{1}u_{3}.$ By Theorem \ref{Tm_BaudonSchrub}
there exists a $2$-aliec $\phi_{a,b}^{0}$ of $G_{0}.$ If $\phi_{a,b}^{0}$ is a
$2$-liec, then $\phi_{a,b,c}=\phi_{a,b}^{0}+\phi_{a,b,c}^{1}$ is a $3$-liec of
$G.$ Otherwise, if $\phi_{a,b}^{0}$ is a proper $2$-aliec of $G_{0}$, we
define the edge coloring $\phi_{a,b,c}$ of $G$ as follows
\[
\phi_{a,b,c}(e)=\left\{
\begin{array}
[c]{ll}%
c & \text{if }e=u_{1}u_{3},\\
\phi_{a,b}^{0}(e) & \text{if }e\in E(G_{0})\backslash\{u_{1}u_{3}\},\\
\phi_{a,b,c}^{1}(e) & \text{if }e\in E(G_{1}).
\end{array}
\right.
\]
It is easily seen that $\phi_{a,b,c}$ is a $3$-liec of $G.$
\end{proof}

Let us now consider unicyclic graphs with larger cycles.

\begin{lemma}
Let $G$ be a unicyclic graph with the unique cycle being of length at least
four. If $G\not \in \mathfrak{T}^{\prime}\mathfrak{,}$ then $\chi_{%
\rm{irr}%
}^{\prime}(G)\leq3.$
\end{lemma}

\begin{proof}
If $G$ is a cycle, then $G\not \in \mathfrak{T}^{\prime}$ implies that $G$ is
an even length cycle and hence admits a $3$-liec. So, we may assume $G$ is not
a cycle, i.e. at least one vertex from the cycle of $G$ is of degree $\geq3$.
Denote the cycle in $G$ by $C=u_{1}u_{2}\cdots u_{g}$ with $g\geq4.$ Without
loss of generality we may assume that $u_{1}$ is the vertex with maximum
degree among vertices from $C.$ We distinguish the following two cases with
respect to $d_{G}(u_{1}).$

\medskip\noindent\textbf{Case 1:} $d_{G}(u_{1})\geq4.$ Let $E_{1}$ denote the
set of all edges incident to $u_{1}$ in $G$ except the edge $u_{1}u_{2}$ and
let $H$ denote the subgraph of $G$ induced by $E_{1}.$ The assumption
$d_{G}(u_{1})\geq4$ implies $d_{H}(u_{1})\geq3,$ so $H$ is a spidey in which
every leg is short. Let $G_{0}$ be the connected component of $G-E_{1}$ which
contains $u_{2}$ and let $G_{1}=G-E(G_{0}).$ Let $G_{1}^{\prime},\ldots
,G_{k}^{\prime}$ be all connected components of $G_{1}-E(H).$ Each
$G_{i}^{\prime}$ is a tree, so Lemma \ref{Lemma_spidey} can be applied to $H$
and $K=G_{i}^{\prime},$ for every $i=1,\ldots,k$. We conclude that there
exists a $3$-liec $\phi_{a,b,c}^{1}$ of $G_{1}$ such that $\phi_{a,b,c}%
^{1}(e)=c$ for every $e\in E(H).$ On the other hand, $G_{0}$ is a shrub rooted
at $u_{1}$ with the root edge $u_{1}u_{2},$ so $G_{0}$ admits $2$-aliec
$\phi_{a,b}^{0}$ according to Theorem \ref{Tm_BaudonSchrub}.

If $\phi_{a,b}^{0}$ is a $2$-liec of $G_{0},$ then $\phi_{a,b,c}=\phi
_{a,b}^{0}+\phi_{a,b,c}^{1}$ is a $3$-liec of $G.$ Otherwise, if $\phi
_{a,b}^{0}$ is a proper $2$-aliec of $G_{0},$ then we define the edge coloring
$\phi_{a,b,c}$ of $G$ in the following way%
\[
\phi_{a,b,c}(e)=\left\{
\begin{array}
[c]{ll}%
c & \text{if }e=u_{1}u_{2},\\
\phi_{a,b}^{0}(e) & \text{if }e\in E(G_{0})\backslash\{u_{1}u_{2}\},\\
\phi_{a,b,c}^{1}(e) & \text{if }e\in E(G_{1}).
\end{array}
\right.
\]
It is easily seen that thus defined $\phi_{a,b,c}$ is a $3$-liec of $G.$

\medskip\noindent\textbf{Case 2:} $d_{G}(u_{1})=3.$ Let $E_{1}$ be the set of
all edges incident to $u_{1}$ in $G$ and $H$ a subgraph of $G$ induced by
$E_{1}.$ Let $G_{0}$ be the connected component of $G-E_{1}$ which contains
$u_{2}$ and let $G_{1}=G-E(G_{0}).$ Similarly as in the previous case, there
exists a $3$-liec $\phi_{a,b,c}^{1}$ of $G_{1}$ such that $\phi_{a,b,c}(e)=c$
for every $e\in E(H).$ Notice that $d_{G}(u_{2})\in\{2,3\},$ since $u_{1}$ is
the vertex with maximum degree among vertices from $C.$ Now we distinguish two
possibilities with regard to $d_{G}(u_{2}).$

If $d_{G}(u_{2})=2,$ then $G_{0}$ is a shrub rooted in $u_{2}$ with the root
edge $u_{2}u_{3}.$ According to Theorem \ref{Tm_BaudonSchrub}, there exists a
$2$-aliec $\phi_{a,b}^{0}$ of $G_{0}.$ If $\phi_{a,b}^{0}$ is $2$-liec of
$G_{0},$ then $\phi_{a,b,c}=\phi_{a,b}^{0}+\phi_{a,b,c}^{1}$ is a $3$-liec of
$G.$ Otherwise, $\phi_{a,b,c}$ defined by%
\[
\phi_{a,b,c}(e)=\left\{
\begin{array}
[c]{ll}%
c & \text{if }e=u_{2}u_{3},\\
\phi_{a,b}^{0}(e) & \text{if }e\in E(G_{0})\backslash\{u_{2}u_{3}\},\\
\phi_{a,b,c}^{1}(e) & \text{if }e\in E(G_{1}).
\end{array}
\right.
\]
is a $3$-liec of $G.$

If $d_{G}(u_{2})=3,$ then consider $G_{0}$ to be a tree rooted at $u_{2}$
which consists of two shrubs $G_{0}^{\prime}$ and $G_{0}^{\prime\prime},$ the
first with the root edge $u_{2}u_{3}$ and the other with the root edge
$u_{2}v_{2},$ where $v_{2}$ is the only neighbor of $u_{2}$ which does not
belong to the cycle $C$. Theorem \ref{Tm_BaudonSchrub} implies that there
exist $2$-aliecs $\phi_{a,b}^{\prime}$ and $\phi_{a,b}^{\prime\prime}$ of
$G_{0}^{\prime}$ and $G_{0}^{\prime\prime},$ respectively. If both $\phi
_{a,b}^{\prime}$ and $\phi_{a,b}^{\prime\prime}$ are a $2$-liec of the
respective shrub, then $\phi_{a,b,c}=\phi_{a,b}^{\prime}+\phi_{b,a}%
^{\prime\prime}+\phi_{a,b,c}^{1}$ is a $3$-liec of $G.$ If both $\phi
_{a,b}^{\prime}$ and $\phi_{a,b}^{\prime\prime}$ are a proper $2$-aliec of the
respective shrub, then $\phi_{a,b,c}=\phi_{a,b}^{\prime}+\phi_{a,b}%
^{\prime\prime}+\phi_{a,b,c}^{1}$ is a $3$-liec of $G.$ The only remaining
possibility is that precisely one of $\phi_{a,b}^{\prime}$ and $\phi
_{a,b}^{\prime\prime},$ say $\phi_{a,b}^{\prime}$, is a proper $2$-aliec of
the respective shrub. In this case we define the coloring $\phi_{a,b,c}^{0}$
of $G_{0}$ as follows
\[
\phi_{a,b,c}^{0}(e)=\left\{
\begin{array}
[c]{ll}%
c & \text{if }e=u_{2}u_{3},\\
\phi_{a,b}^{\prime}(e) & \text{if }e\in E(G_{0}^{\prime})\backslash
\{u_{2}u_{3}\},\\
\phi_{a,b}^{\prime\prime}(e) & \text{if }e\in E(G_{0}^{\prime\prime}).
\end{array}
\right.
\]
Since $\left\vert V(C)\right\vert \geq4,$ it is easily seen that $\phi
_{a,b,c}=\phi_{a,b,c}^{0}+\phi_{a,b,c}^{1}$ is a $3$-liec of $G.$
\end{proof}

The previous two lemmas yield the following result.

\begin{theorem}
\label{Tm_unicyclic}Let $G$ be a unicyclic graph. If $G\not \in
\mathfrak{T}^{\prime}\mathfrak{,}$ then $\chi_{%
\rm{irr}%
}^{\prime}(G)\leq3.$
\end{theorem}

A natural question that arises is whether the bound $\chi_{%
\rm{irr}%
}^{\prime}(G)\leq3$ is tight, i.e. are there colorable unicyclic graphs which
are not $2$-colorable. The family of cycles of length $4k+2$ are such graphs,
but this family is not an isolated case, there exist other unicyclic graphs
which require three colors, for example the graph from Figure
\ref{Fig_unicyclic}. One can assure infinitely many such graphs for example by
taking longer threads of suitable parity in the given graph.

\begin{figure}[h]
\begin{center}
\includegraphics[scale=0.8]{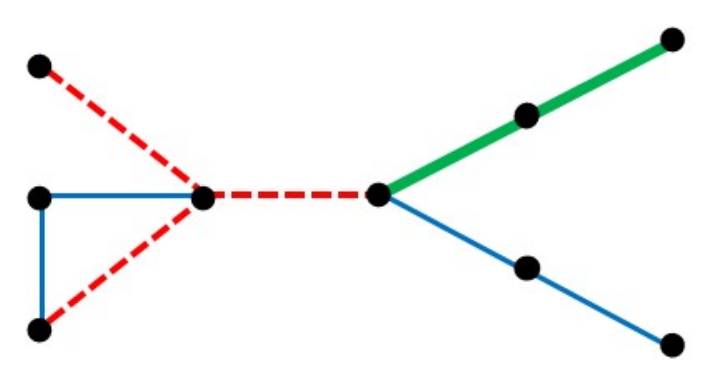}
\end{center}
\caption{A colorable unicyclic graph distinct from cycle which requires $3$
colors for locally irregular edge coloring.}%
\label{Fig_unicyclic}%
\end{figure}

\section{Cacti with vertex disjoint cycles}

In this section we will extend the result from the previous section to cacti
with vertex disjoint cycles. We will also show that the result does not extend
to all cacti by providing an example of a cactus graph with four cycles which
is colorable, but requires $4$ colors for a locally irregular edge coloring.
This establishes that Conjecture \ref{Con_nonTareColorable} does not hold in
general. We first need to introduce several useful notions in order to deal
with cacti.

Let $G$ be a cactus graph with at least two cycles, let $C$ be a cycle in $G$
and let $u$ be a vertex from $C.$ We say that $u$ is a \emph{root vertex} of
$C$ if the connected component of $G-E(C)$ which contains $u$ is a cyclic
graph. A cycle $C$ of $G$ is a \emph{proper end-cycle} if $G-V(C)$ contains at
most one cyclic connected component. Every cactus graph with vertex disjoint
cycles contains at least two proper end-cycles, given it is not a unicyclic graph.

\begin{theorem}
Let $G$ be a cactus graph with vertex disjoint cycles. If $G\not \in
\mathfrak{T}^{\prime}\mathfrak{,}$ then $\chi_{%
\rm{irr}%
}^{\prime}(G)\leq3.$
\end{theorem}

\begin{proof}
The proof is by induction on the number of cycles in $G.$ If $G$ is a
unicyclic graph, then the claim holds by Theorem \ref{Tm_unicyclic}. Assume
that the claim holds for all cacti with less than $p$ cycles, where $p\geq2$.
Let $G$ be a cactus graph with $p$ cycles. We will show that $G$ admits a
$3$-liec and this will establish the claim of the theorem. Let $C$ be a proper
end-cycle of $G$, $u_{1}$ the root vertex of $C,$ and $v$ the only neighbor of
$u_{1}$ which belongs to the cyclic component of $G-u_{1}v.$ Denote the other
neighbors of $u_{1}$ by $u_{2},\ldots,u_{k}$ so that $u_{2}$ and $u_{3}$
belong to the cycle $C$. In what follows, we distinguish two cases.

\medskip\noindent\textbf{Case 1:} $d_{G}(u_{1})=3$. Let $G_{1}$ be the
connected component $G-u_{1}v$ which does not contain $v$ and let
$G_{0}=G-E(G_{1}).$ Let $G_{0}^{\prime}=G_{0}+u_{1}u_{2}$ and $G_{1}^{\prime
}=G_{1}-u_{1}u_{2}.$

Suppose first that $G_{0}^{\prime}$ is colorable. Then it admits a $3$-liec
$\phi_{a,b,c}^{0}$ where the edges $u_{1}u_{2}$ and $u_{1}v$ must be colored
by a same color, say color $c$. Notice that $G_{1}^{\prime}$ is a shrub rooted
at $u_{1}$ with the root edge $u_{1}u_{3}.$ By Theorem \ref{Tm_BaudonSchrub},
$G_{1}^{\prime}$ admits $2$-aliec $\phi_{a,b}^{1}$. If $\phi_{a,b}^{1}$ is a
$2$-liec of $G_{1}^{\prime}$, then $\phi_{a,b,c}^{0}+\phi_{a,b}^{1}$ is a
$3$-liec of $G.$ Otherwise, if $\phi_{a,b}^{1}$ is a proper $2$-aliec of
$G_{1},$ then the restriction of $\phi_{a,b}^{1}$ to $G_{1}^{\prime\prime
}=G_{1}^{\prime}-u_{1}u_{3}$ is a $2$-liec of that graph. Notice that
$G_{0}^{\prime\prime}=G_{0}^{\prime}+u_{1}u_{3}$ does not belong to
$\mathfrak{T}^{\prime}$, so it is colorable and by induction hypothesis it
admits a $3$-liec $\phi_{a,b,c}^{\prime\prime}$ in which edges $u_{1}u_{2}$
and $u_{1}u_{3}$ must be colored by a same color (say color $c$) since
$d_{G}(u_{1})=3.$ Now we infer that%
\[
\phi_{a,b,c}(e)=\left\{
\begin{array}
[c]{cc}%
\phi_{a,b}^{1}(e) & \text{if }e\in E(G_{1}^{\prime\prime}),\\
\phi_{a,b,c}^{\prime\prime}(e) & \text{if }e\in E(G_{0}^{\prime\prime}),
\end{array}
\right.
\]
is a $3$-liec of $G$.

Suppose now that $G_{0}^{\prime}$ is not colorable. Assume first $G_{1}$ is
not colorable. Notice that $G_{1}$ is a unicyclic graph, so if the cycle of
$G_{1}$ is a triangle, then the assumption that $G_{0}^{\prime}$ and $G_{1}$
are not colorable would imply $G\in\mathfrak{T}$, a contradiction. Otherwise,
if $G_{1}$ is a unicyclic graph on a larger cycle, then it is not colorable
only if it is an odd length cycle. In this case let $w$ be the only neighbor
of $u_{3}$ distinct from $u_{1},$ let $G_{0}^{\prime\prime}=G_{0}^{\prime
}+\{u_{1}u_{3},u_{3}w\}$ and $G_{1}^{\prime\prime}=G-E(G_{0}^{\prime\prime}).$
Notice that by induction hypothesis $G_{0}^{\prime\prime}$ is colorable and
admits a $3$-liec $\phi_{a,b,c}^{0}$ for which we may assume $\phi_{a,b,c}%
^{0}(u_{3}w)=a$ and $\phi_{a,b,c}^{0}(u_{1}u_{2})\in\{a,b\}.$ Also, notice
that $G_{1}^{\prime\prime}$ is an even length path, so it admits a $2$-liec
$\phi_{b,c}^{1}$ where we may assume that the edge of $G_{1}^{\prime\prime}$
incident to $u_{2}$ is colored by $c.$ Then $\phi_{a,b,c}=\phi_{a,b,c}%
^{0}+\phi_{b,c}^{1}$ is a $3$-liec of $G.$

Suppose now that $G_{1}$ is colorable. Since $G_{0}^{\prime}$ is not
colorable, the edge $u_{1}u_{2}$ of $G_{0}^{\prime}$ must belong to an even
length path hanging at a vertex of a triangle in $G_{0}^{\prime},$ so the
graph $G_{0}=G_{0}^{\prime}-u_{1}u_{2}$ contains an odd length path hanging at
a vertex of a triangle, which means $G_{0}\not \in \mathfrak{T}^{\prime},$ so
it is colorable. Therefore, by induction hypothesis $G_{0}$ admits a $3$-liec
$\phi_{a,b,c}^{0}.$ Since $u_{1}$ is a leaf in $G_{0}$, we may assume that
$\phi_{a,b,c}^{0}(u_{1}v)=c.$ By Theorem \ref{Tm_unicyclic}, $G_{1}$ admits
$3$-liec $\phi_{a,b,c}^{1}.$ Since the degree of $u_{1}$ in $G_{1}$ equals
two, we may assume that the colors of edges $u_{1}u_{2}$ and $u_{1}u_{3}$ are
from $\{a,b\}.$ Therefore, $\phi_{a,b,c}^{0}+\phi_{a,b,c}^{1}$ is a $3$-liec
of $G$.

\medskip\noindent\textbf{Case 2:} $d_{G}(u_{1})\geq4.$ Let $H$ be the subgraph
of $G$ induced by the set of all edges incident to $u_{1}$ in $G.$ Denote the
connected components of $G-E(H)$ in the following way, let $G_{0}^{\prime}$ be
the component which contains $v$ and $G_{1}^{\prime}$ the component which
contains $u_{2}$ and $u_{3}$. Also, let $G_{1}=G_{1}^{\prime}+u_{1}u_{2}$ and
$G_{0}=G_{0}^{\prime}+u_{1}v.$ We may assume $G_{0}$ is colorable, as
otherwise $G$ would contain a proper end-cycle which is a triangle with the
root vertex of degree $3$, which would reduce to the previous case. Let
$G_{2}=G-(E(G_{0})\cup E(G_{1})\cup E(H))$ and $H^{\prime}=H-\{u_{1}%
u_{2},u_{1}v\}.$

Suppose first that the tree $H^{\prime}+G_{2}$ is not colorable. This implies
that it is an odd length path. Notice that $H^{\prime}+G_{2},$ as a shrub
rooted at $u_{3},$ admits a proper $2$-aliec $\phi_{a,b}^{2},$ and since it is
proper we have $\phi_{a,b}^{2}(u_{1}u_{4})=b.$ Since $G_{0}$ is colorable, by
induction hypothesis it admits a $3$-liec $\phi_{a,b,c}^{0},$ where we may
assume $\phi_{a,b,c}^{0}(u_{1}v)=a.$ Since $G_{1}$ is a shrub rooted at
$u_{1}$ with the root vertex $u_{1}u_{2},$ it admits a $2$-aliec $\phi
_{a,b}^{1}.$ If $\phi_{a,b}^{1}$ is a proper $2$-aliec of $G_{1}$, then
\[
\phi_{a,b,c}(e)=\left\{
\begin{array}
[c]{ll}%
c & \text{if }e=u_{1}u_{2}\text{ or }u_{1}u_{3},\\
(\phi_{a,b,c}^{0}+\phi_{a,b}^{1}+\phi_{a,b}^{2})(e) & \text{if }e\in
E(G)\backslash\{u_{1}u_{2},u_{1}u_{3}\},
\end{array}
\right.
\]
is a $3$-liec of $G.$ Otherwise, if $\phi_{a,b}^{1}$ is a $2$-liec of $G_{1},$
then let us consider the graph $G_{0}^{\prime\prime}=G_{0}+u_{1}u_{3}.$ It is
colorable by the same argument as $G_{0},$ so it admits a $3$-liec
$\phi_{a,b,c}^{\prime\prime}$ in which $u_{1}u_{3}$ and $u_{1}v$ must be
colored by a same color, say $c$. Then%
\[
\phi_{a,b,c}(e)=\left\{
\begin{array}
[c]{ll}%
\phi_{a,b,c}^{\prime\prime}(e) & \text{if }e\in E(G_{0}^{\prime\prime}),\\
(\phi_{a,b}^{1}+\phi_{a,b}^{2})(e) & \text{if }e\in E(G)\backslash
E(G_{0}^{\prime\prime}),
\end{array}
\right.
\]
is a $3$-liec of $G.$

Suppose now that $H^{\prime}+G_{2}$ is a colorable tree, so it admits a
$3$-liec $\phi_{a,b,c}^{2}.$ We may assume that $\phi_{a,b,c}^{2}(e)=c$ for
every $e\in E(H^{\prime}),$ as this follows either from $d_{H^{\prime}}%
(u_{1})=2$ or from Lemma \ref{Lemma_spidey} applied to $H^{\prime}$ as a
spidey and every connected component of $G_{2}$ as $K$. As for $G_{0},$ recall
that it is colorable, so by induction hypothesis, it has a $3$-liec
$\phi_{a,b,c}^{0}(e)$. Since $u_{1}$ is a leaf in $G_{0},$ we may assume
$\phi_{a,b,c}^{0}(u_{1}v)=a.$\ Let us now consider the graph $G_{1}.$ Recal
that it is a shrub rooted at $u_{1}$ with the root edge $u_{1}u_{2}.$ Hence,
by Theorem \ref{Tm_BaudonSchrub} the graph $G_{1}$ admits a $2$-aliec
$\phi_{a,b}^{1}.$ If $\phi_{a,b}^{1}$ is a $2$-liec of $G_{1},$ then
$\phi_{a,b,c}^{0}+\phi_{b,a}^{1}+\phi_{a,b,c}^{2}$ is a $3$-liec of $G.$
Otherwise, we define $H^{\prime\prime}=H^{\prime}+u_{1}u_{2}$, and notice that
$H^{\prime\prime}$ is a spidey. According to Lemma \ref{Lemma_spidey}, the
graph $G_{2}^{\prime\prime}=H^{\prime\prime}+G_{2}$ admits a $3$-liec
$\phi_{a,b,c}^{\prime\prime}$ such that $\phi_{a,b,c}^{\prime\prime}(e)=c$ for
every $e\in E(H^{\prime\prime}).$ We conclude that
\[
\phi_{a,b,c}(e)=\left\{
\begin{array}
[c]{ll}%
\phi_{a,b,c}^{\prime\prime}(e) & \text{if }e\in E(G_{2}^{\prime\prime}),\\
(\phi_{a,b,c}^{0}+\phi_{a,b}^{1})(e) & \text{if }e\in E(G)\backslash
E(G_{2}^{\prime\prime}).
\end{array}
\right.
\]
is a $3$-liec of $G.$
\end{proof}

Let us now consider the so called \emph{bow-tie} graph $B$ shown in Figure
\ref{Fig_bowGraph}. This is a cactus graph with four cycles, but in which
cycles are not vertex disjoint. This graph is colorable and admits $4$-liec
shown in Figure \ref{Fig_bowGraph}, but it does not admit $k$-liec for
$k\leq3$ since the two end-vertices of the cut edge must have the degree three
in the color of that edge. Hence, for the bow-tie graph $B$ it holds that
$\chi_{%
\rm{irr}%
}^{\prime}(B)=4.$ We conclude that Conjecture \ref{Con_nonTareColorable} does
not hold in general.

\begin{figure}[h]
\begin{center}
\includegraphics[scale=0.8]{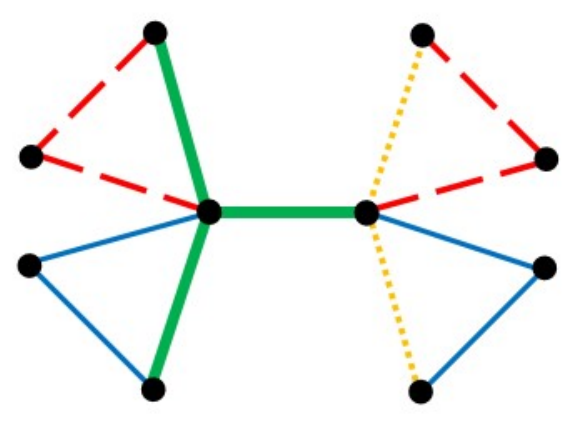}
\end{center}
\caption{The bow-tie graph $B$ and a $4$-liec of it.}%
\label{Fig_bowGraph}%
\end{figure}

The consideration of the bow-tie graph gives rise to the following questions:
are there any other graphs for which Conjecture \ref{Con_nonTareColorable}
does not hold, do all colorable cacti admit a $4$-liec, what is the thight
upper bound on $\chi_{%
\rm{irr}%
}^{\prime}(G)$ of general graphs? We believe the following conjectures holds,
which is a weaker form of the Local Irregularity Conjecture.

\begin{conjecture}
Every connected graph $G$ which does not belong to $\mathfrak{T}^{\prime}$
satisfies $\chi_{%
\rm{irr}%
}^{\prime}(G)\leq4.$
\end{conjecture}

\bigskip\noindent\textbf{Acknowledgments.}~~Both authors acknowledge partial
support of the Slovenian research agency ARRS program\ P1-0383 and ARRS
project J1-1692. The first author also the support of Project
KK.01.1.1.02.0027, a project co-financed by the Croatian Government and the
European Union through the European Regional Development Fund - the
Competitiveness and Cohesion Operational Programme.

\end{document}